\documentclass[10pt]{article}
\font\smallit=cmti10

\usepackage{enumitem}
\usepackage{amssymb,latexsym,amsmath,epsfig,amsthm} 
\makeatletter

\renewcommand\section{\@startsection {section}{1}{\z@}
{-30pt \@plus -1ex \@minus -.2ex}
{2.3ex \@plus.2ex}
{\normalfont\normalsize\bfseries\boldmath}}

\renewcommand\subsection{\@startsection{subsection}{2}{\z@}
{-3.25ex\@plus -1ex \@minus -.2ex}
{1.5ex \@plus .2ex}
{\normalfont\normalsize\bfseries\boldmath}}

\renewcommand{\@seccntformat}[1]{\csname the#1\endcsname. }

\newcommand\blfootnote[1]{
\begingroup
\renewcommand\thefootnote{}\footnote{#1}
\addtocounter{footnote}{-1}
\endgroup
}

\makeatother

\newtheorem{theorem}{Theorem}
\newtheorem{lemma}{Lemma}
\newtheorem{proposition}{Proposition}
\newtheorem{corollary}{Corollary}

\theoremstyle{definition}

\newtheorem{remark}{Remark}
\newtheorem{example}{Example}

\begin{document}
\begin{figure}
\vspace*{-62pt}

\vspace{.75in}
\end{figure}
\vspace*{-45pt}

\vskip 3pt 	

%\title{A Prime-Power Dichotomy for Nekrasov--Okounkov Hook Lengths and $t$-Core Partitions}
\begin{center}
\uppercase{\bf \boldmath A Prime-Power Dichotomy for Nekrasov--Okounkov Hook Lengths  and} {$t$\uppercase{\bf \boldmath-Core Partitions }}
\vskip 20pt

{\bf Meenu Sharma \blfootnote{%
  2020 Mathematics Subject Classification: Primary 11P83, 11P84; Secondary 05A17, 11F33.\\
  Keywords: Nekrasov--Okounkov hook length formula, $t$-core partitions, partition congruences, hook lengths, formal power series, prime powers.}}\\
{\smallit Department of Mathematics,  The University of Hong Kong, Pokfulam, Hong Kong}\\ 
{\tt sharmameenuphd@gmail.com}\\ 
\end{center}

\vskip 30pt

\centerline{\bf Abstract}
\noindent
Let $a_t(n)$ denote the coefficients of the Nekrasov--Okounkov hook length generating function $F_{t^2}(x)$, and let $b_t(n)$ denote the coefficients of the $t$-core partition generating function $C_t(x)$. We prove a sharp prime-power dichotomy: for every integer $t\ge 2$, the congruence
$$
a_t(n)\equiv b_t(n)\pmod t
$$
holds for all $n\ge 0$ if and only if $t$ is a prime power. When $t=p$ is prime, the congruence strengthens to modulo $p^2$. For $t=6$, and more generally for every integer that is not a prime power, the congruence fails. This identifies prime powers as the exact moduli for which these two partition-theoretic coefficient sequences are arithmetically equivalent.\\
\pagestyle{myheadings}

\thispagestyle{empty}
\baselineskip=12.875pt
\vskip 30pt

\section{Introduction and Statement of Results}
The arithmetic and combinatorial properties of integer partitions and their
hook lengths have long played a central role at the intersection of algebraic
combinatorics, representation theory, and modular forms. For a partition
$\lambda$, its Young diagram consists of $|\lambda|$ cells, each equipped
with a positive integer known as its \emph{hook length} $h$. In 2006,
Nekrasov and Okounkov \cite{Nek-Oko} introduced a foundational identity
expressing complex powers of the Euler product $\prod_{n=1}^{\infty}(1-x^n)$
as a multiset product over partition hook lengths. For a partition $\lambda$,
denote by $|\lambda|$ its size and by $\mathcal{H}(\lambda)$ the multiset of
hook lengths of its Young diagram. Define
\begin{equation}\label{hooklenth}
F_z(x) := \sum_{\lambda} x^{|\lambda|}
\prod_{h\in \mathcal{H}(\lambda)}\left( 1-\frac{z}{h^2}\right)
= \prod_{n=1}^{\infty}(1-x^n)^{z-1},
\end{equation}
where the outer sum ranges over all integer partitions $\lambda$ and
$z\in \mathbb{C}$ is an arbitrary parameter. Han \cite{GNHan} subsequently
provided a remarkable elementary combinatorial proof of \eqref{hooklenth}.

Concurrently, $t$-core partitions (i.e., partitions containing no hook
lengths divisible by $t$) are fundamental structures in representation
theory. Under Nakayama's conjecture, $p$-cores label the $p$-blocks of
symmetric group representations over fields of characteristic $p$.
Garvan, Kim, and Stanton \cite{FDD} extensively developed the combinatorial
properties and generating functions of these $t$-cores. The generating
function for the number $b_t(n)$ of $t$-core partitions of $n$ is given
by the product quotient
\begin{equation}
C_t(x) = \prod_{n=1}^{\infty} \frac{(1-x^{tn})^t}{1-x^n}
= \sum_{n=0}^{\infty} b_t(n)\, x^n.
\end{equation}
The arithmetic of $b_t(n)$ was revolutionized by Granville and Ono
\cite{Granville-Ono}, who proved the $t$-core partition conjecture---for every
$t\ge 4$, every nonnegative integer $n$ possesses at least one $t$-core
partition, i.e. $b_t(n)>0$ and established the Ramanujan-type congruences
$$
b_5(5n+4)\equiv 0\pmod 5,\qquad
b_7(7n+5)\equiv 0\pmod 7,\qquad
b_{11}(11n+6)\equiv 0\pmod{11}.
$$
These results played a key role in the classification of defect-zero
$p$-blocks of finite simple groups.

In 2011, Han and Ono \cite{Han-Ono} forged a striking connection between these
two domains by evaluating the Nekrasov--Okounkov series at $z=3^2=9$.
Letting $F_9(x)=\sum_{n=0}^{\infty} a_3(n)x^n$, they observed that $a_3(n)$
vanishes if and only if the $3$-core partition function $b_3(n)$ vanishes.
Using deep analytic results on weight $4$ modular forms on $\Gamma_0(9)$
with complex multiplication by $\mathbb{Q}(\sqrt{-3})$, they proved the
qualitative zero-equivalence theorem
$$
a_3(n)=0 \iff b_3(n)=0 \qquad (n\ge 0).
$$

In this paper we go beyond zero-equivalence and ask: do the coefficients
$a_t(n)$ of $F_{t^2}(x)$ and $b_t(n)$ of $C_t(x)$ satisfy arithmetic
congruences modulo $t$? Does such a relationship hold for prime $t=p$,
for prime powers $t=p^a$, or for general composite integers $t$? And can
these arithmetic relations be established via elementary $p$-adic formal
power series methods, independently of CM modular form machinery? Our main
findings are summarized as follows.

Our main result is the following complete characterization.

\begin{theorem}[Complete Characterization]\label{compositecase}
Let $t\ge 2$ be a positive integer. The following are equivalent:
\begin{enumerate}[label=(\roman*)]
    \item $a_t(n)\equiv b_t(n)\pmod t$ for all $n\ge 0$;
    \item $t$ is a prime power.
\end{enumerate}
\end{theorem}

The proof of Theorem~\ref{compositecase} is divided into two independent
directions. The implication $(\mathrm{ii})\Rightarrow(\mathrm{i})$ is
established by the following two results, which treat the prime case and
then extend it to arbitrary prime powers.

\begin{proposition}\label{Primecase}
Let $t=p$ be any prime. Then for all $n\ge 0$,
$$
a_p(n)\equiv b_p(n)\pmod{p^2}.
$$
In particular, $a_p(n)\equiv b_p(n)\pmod p$.
\end{proposition}

\begin{proposition}[Prime Power Case]\label{Primepowercase}
Let $t=p^a$ be any prime power with $a\ge 1$. Then for all $n\ge 0$,
$$
a_t(n)\equiv b_t(n)\pmod{p^a}.
$$
\end{proposition}

Propositions~\ref{Primecase} and~\ref{Primepowercase} together prove the
direction $(\mathrm{ii})\Rightarrow(\mathrm{i})$ of
Theorem~\ref{compositecase}. The converse direction
$(\mathrm{i})\Rightarrow(\mathrm{ii})$ is established independently in
Section~\ref{sec:failure}, where we show that the congruence fails for every
integer that is not a prime power.
\section{Proofs of Propositions \ref{Primecase} and \ref{Primepowercase}}
\begin{lemma}[Freshman's Dream for Formal Power Series]\label{Freshman}
 Let $p$  be a prime and $f(x) \in \mathbb{Z}[[x]]$  with $f(0) =1$. Then
 $$f(x)^p \equiv f(x^p) \pmod p \quad \text{in}\; \mathbb{Z}[[x]].$$
\end{lemma}
\begin{proof}
    Write $f(x) = 1 + \sum_{k\geq 1} c_k x^k$. By the multinomial theorem, all cross terms in $f(x)^p$ carry a binomial coefficient $\binom{p}{m}$, where $m$ is a multi-index with $|m| = p$. For $p$ prime, $p\mid \binom{p}{m}$ in all such cases. The surviving terms are $\sum_{k}c_k^p x^{pk} \equiv \sum_{k} c_k x^{pk} \equiv f(x^p) \pmod{p}$, where $c_k^p \equiv c_k \pmod p$ by Fermat's little theorem.
\end{proof}
\begin{corollary}\label{Coro}
    Set $A(x) = \prod_{n\geq 1} (1-x^n)$ and $B_k(x) = \prod_{n\geq 1} (1-x^{np^k})$ for $k\geq 0$. Then for each $k\geq 0$,
    $$B_k(x)^p \equiv B_{k+1}(x) \pmod p.$$
\end{corollary}
We now prove Proposition \ref{Primecase} using Lemma \ref{Freshman}.
\begin{proof}[Proof of Proposition \ref{Primecase}]
Let $p$ be a prime, write $A = A(x)$ and $B = B_1(x) = \prod_{n\geq 1}(1-x^{np})$. By Lemma \ref{Freshman}, we have 
$$A^p \equiv B \pmod p.$$
Write $A^p = B+p.R$,  for some $R\in Z[[x]]$. So we have 
\begin{align*}
 A^{p^2} &= (A^p)^p =(B+pR)^p = \sum_{k=0}^{p}\binom{p}{k}B^{p-k}(pR)^k \\  &= B^p +p.B^{p-1}.pR +\binom{p}{2}(pR)^2B^{p-2}+...
\end{align*}
For $k =1$, the coefficient is $\binom{p}{1}p =p^2$. For $k\geq 2,$ the coefficient contains $p^k\geq p^2.$ Therefore, every term with $k\geq 1$ is divisible by $p^2$, giving
\begin{equation}\label{APBP}
 A^{p^2} \equiv B^p \pmod{p^2}.   
\end{equation}

We observe that 
\begin{equation}\label{FPCP}
 F_{p^2}(x) = A^{p^2 -1} = \frac{A^{p^2}}{A}, \quad C_p(x) = \frac{B_p}{A} =\frac{\prod_{n}(1-x^{np})^p}{\prod_n (1-x^n)}.   
\end{equation}

Since $A^{p^2}\equiv B^p \pmod{p^2}$ in $\mathbb{Z}[[x]]$ and $A$ has constant term $1$ (so division is well-defined in $\mathbb{Z}[[x]]$).
Using Equations \eqref{APBP} and \eqref{FPCP}, we get 
\begin{equation}\label{FPP}
    F_{p^2}(x) =\frac{A^{p^2}}{A} \equiv \frac{B^p}{A} = C_p(X) \pmod{p^2}.
\end{equation}
After comparing coefficients, we have 
$$a_p(n) \equiv b_p(n) \pmod {p^2} \quad \text{for all} \; n\geq 0.$$
\end{proof}
\begin{remark}
    For prime $t=p$, the congruence holds for $p^2$, which is stronger than our stated $\pmod{p}$. Verified numerically for $p=2,3,5.$
\end{remark}

We need the following lemma to prove Proposition \ref{Primepowercase}.
\begin{lemma}[Lifting Lemma] With $B_k$ as above, for all $k\geq 0$ and $j\geq 1$, we have 
$$B_k^{p^{2j}} \equiv B_{k+j}^{p^j} \pmod{p^j} $$
\begin{proof}
    By induction on $j$.\\
    \textit{Base case:} For $j=1$, we need $B_k^{p^2} \equiv B_{k+1}^p \pmod p.$ By Corollary \ref{Coro}, we know that  $B_{k}^p\equiv B_{k+1} \pmod p.$ Write $B_k^{p} = B_{k+1} + pS.$ Then 
    $$B_k^{p^2} = (B_k^p)^p = (B_{k+1} +pS)^p \equiv B_{k+1}^p \pmod {p^2}$$
    and reducing modulo $p$ gives $B_{k}^{p^2} \equiv B_{k+1}^p \pmod{p}.$\\
    \textit{Inductive step:} Assume $B_{k}^{p^{2j}} \equiv B_{k+j}^{p^j} \pmod{p^j}$. Write $B_{k}^{p^{2j}} = B_{k+j}^{p^j} +p^jU.$
    Then
    $$B_k^{p^{2(j+1)}} = (B_{k+j}^{p^j} +p^j U)^{p^2} = \sum_{l=0}^{p^2} \binom{p^2}{l} (B_{k+j}^{p^j})^{p^2 -l} (p^j U)^l.$$
    The $l=0$ term is $B_{k+j}^{p^{j+2}}$. For $l\geq 1,$ $\binom{p^2}{l}p^{jl} \geq p^{j+1}.$ Hence $B_k^{p^{2(j+1)}} \equiv B_{k+j}^{p^{j+2}} \pmod {p^{j+1}}.$
    The base case proof, applied to $B_{k+j}$ in place of $B_k$, establishes the stronger congruence $B_{k+j}^{p^2} \equiv B_{k+j+1}^p \pmod {p^2}$; write $B_{k+j}^{p^2} = B_{k+j +1}^{p} +p^2 W$ for some $ W\in \mathbb{Z}[[x]].$ Then
    \begin{align*}
       B_{k+j}^{p^{j+2}} &= (B_{k+j}^{p^2})^{p^j} = (B_{k+j+1}^{p} +p^2 W)^{p^j} \\
       & B_{k+j+1}^{p^{j+1}} +\binom{p^j}{1} (B_{k+j+1})^{p^j-1}.p^2 W+...
    \end{align*}
    The $l=1$ cross-term carries $p^j.p^2 =p^{j+2}$, all higher terms carry at least $p^{2j} \geq p^{j+1}$ (for $j\geq 1)$. Hence $B_{k+j}^{p^{j+2}} \equiv B_{k+j+1}^{p^{j+1}} \pmod{p^{j+1}}.$
   
\end{proof}
    
\end{lemma}

\begin{proof}[Proof of Proposition \ref{Primepowercase} ]
Apply the lifting Lemma with $k=0$ and $j=a$   
$$A^{p^{2a}} \equiv B_a^{p^a} \pmod{p^a}. $$
Here $B_a = \prod_{n} (1-x^{np^a}) = \prod_{n} (1-x^{nt})$ and $p^a = t,$ so $ B_a^{p^a} = \left(\prod_{n}(1-x^{tn})\right)^t.$ Dividing by $A,$
$$F_{t^2}(x) = \frac{A^{t^2}}{A} \equiv \frac{B_a^{t}}{A} = C_t(x) \pmod{t}.$$
Hence $a_t(n) \equiv b_t(n) \pmod{t}$ for all $n\geq 0.$
\end{proof}

\section{Failure for Non-Prime-Powers}\label{sec:failure}
To prove the converse direction of Theorem~\ref{compositecase}, we establish the following lemma.

\begin{lemma}\label{Pro-p}
Let $p$ be a prime and $m\geq 1$ with $\gcd(m,p) = 1$. Then the map $f\to f^m$  is injective on $1 +y \mathbb{F}_p[[y]]$. 
\end{lemma}
\begin{proof}
    The group $G = 1 +y\mathbb{F}_p[[y]]$ under multiplication is a pro-$p$ group, every element has order a power of $p$ (this follows because the logarithm map shows $G\cong (y\mathbb{F}_p[[y]], +)$ as a topological groups via $\log(1+g) = g-g^2/2+...$, which in characteristic $p$  is a group of exponent $p^{\infty}$). Since $\gcd(m,p)=1,$ the map $x\to x^m$ is an automorphism of any pro-$p$ abelian group, hence injective on $G$.
\end{proof}
\begin{proof}[Proof of Theorem \ref{compositecase}]The implication $ \mathrm{(ii)}\Rightarrow\mathrm{(i)}$ follows from Propositions \ref{Primecase} and \ref{Primepowercase}. 
We now prove \(\mathrm{(i)}\Rightarrow\mathrm{(ii)}\) by contraposition. 
Suppose \(t\) is not a prime power. Write
$$t=p^am, \quad p \;\text{prime}, \quad a\geq 1,\quad \gcd(p,m)=1, \quad m>1.$$
Since $\gcd(p,m) =1$ and $m\geq 1$, the integer $m$ is not a power of $p$. By Kummer's theorem \cite{Kummer1852}, there exists $k_0$  with $1\leq k_0\leq m-1$ such that 
$$p\nmid \binom{m}{k_0}$$

We will show $A^{t^2} \not \equiv B^t \pmod{p}$ (which implies $F_{t^2}(x) \not \equiv C_t(x)\pmod{p}$ and hence $\pmod t$, since $p\mid t).$
Recall $A= \prod_{n\geq 1}(1-x^n)$ and $B= \prod_{n\geq 1}(1-x^{tn}).$ Applying Lemma \ref{Freshman} iteratively $2a$ times to $A$, we have
$$A^{2a} \equiv A(x^{p^{2a}}) \pmod p, \quad \text{so} \quad A^{t^2}= A^{p^{2a}m^2} = (A^{p^{2a})m^2} \equiv A(x^{p^{2a}})^{m^2} \pmod p.$$
For $B^t = \prod_{n\geq 1}(1-x^{p^a mn})^{{p^a}m}$, each factor satisfies
$$(1-x^{p^a m n})^{p^a m} \equiv (1-x^{2p^a m n})^m \pmod p,$$
Set $y=x^{p^{2a}}$. The congruence $A^{t^2} \equiv B^t \pmod{p}$ becomes
\begin{equation}\label{y^mn}
 \prod_{n\geq 1}(1-y^n)^{m^2} \equiv \prod_{n\geq 1} (1-y^{mn})^m \pmod{p}.   
\end{equation}
Both sides of \eqref{y^mn} lie in $1+y\mathbb{F}_p[[y]]$. Since $\gcd(m,p) =1$, Lemma \ref{Pro-p} says $f^m = g^m \implies f =g$. Writing \eqref{y^mn}
 as $\left(\prod_{n}(1-y^n)^m\right)^m \equiv \left( \prod(1-y^{mn}))\right)^m$, we may cancel the $m$-th power to obtain:
 \begin{equation}\label{eqy^n}
     \prod_{n\geq 1} (1-y^n)^m \equiv \prod_{n\geq 1}(1-y^{mn})\pmod{p}.
 \end{equation}
 We now show that the reduced congruence \eqref{eqy^n} is impossible. Since $m>1$ and $\gcd(m,p) =1,$ it suffices to compare the coefficients of the linear form $y^1$ modulo $p$.\\
 Applying the coefficient extraction operator $[y^1]$ to the left side of \eqref{eqy^n}, we obtain 
 $$[y^1]\prod_{n\geq 1}(1-y^n)^m.$$
For every $n\geq 2$, the factor $(1-y^n)^m = 1-my^n + O(y^n+1)$ has no contribution to the coefficient of $y^1$, because its lowest non-zero degree is $n\geq 2$. Hence, the only factor corresponding to $n=1$ survives. Consequently,
$$[y^1] \prod_{n\geq 1} (1-y^n)^m = [y^1](1-y)^m = -m.$$
On the right-hand side of \eqref{eqy^n}, we have
$$[y^1] \prod_{n\geq 1}(1-y^{mn}).$$
Since $m>1$, every factor $(1-y^{mn})$ has lowest degree at least $m\geq 2$. Thus, the coefficient of $y^1$ is trivially zero:
$$[y^1] \prod_{n\geq 1} (1-y^{mn}) =0.$$
Combining these two evaluations yields
$$[y^1] \left(\prod_{n\geq 1}(1-y^n)^m - \prod_{n\geq 1} (1-y^{mn} )\right) = -m$$
Given that $\gcd(m,p) =1$, we have $m\not\equiv 0\pmod p$, and therefore
this contradicts the congruence in \eqref{eqy^n}. Hence, the congruence \eqref{eqy^n} cannot hold modulo $p$. Tracing back through the substitutions, this shows that the original congruence $A^{t^2} \not \equiv B^t \pmod{p}$ must fail. Consequently, $F_{t^2}(x) \not \equiv C_t(x) \pmod{p}$, establishing the existence of some $n\geq 1$ such that $a_t(n) \not\equiv b_t(n) \pmod t.$
\end{proof}
\section{Precise Location of the First Failure}
\begin{corollary}\label{corofailure}
    Let $t=p^a m$ with $\gcd(p,m)= 1$ and $m>1$. Then 
    $$[x^{p^{2a}}] (F_t^2(x) - C_t(x)) \not\equiv 0\pmod p.$$
    Consequently, the first index $n$ for which $a_t(n) \not\equiv b_t(n) \pmod t$ satisfies $n\leq p^{2a}$.
\end{corollary}

\begin{example}[t=6] Let
$$F_{36}(x) = \prod_{n\geq1} (1-x^n)^{36} =\sum_{n\geq 0} a_6(n) x^n, \quad C_6(x) = \prod_{n\geq 1}(1-x^{6n})^6 = \sum_{n\geq 0} b_6(n) x^n.$$
\emph{Perspective 1} $(p=2)$. Write $6=2^1\cdot3$, so $a=1, m=3.$ Corollary \ref{corofailure} predicts failure at $n\leq p^{2a} =4.$ Indeed,
$$a_6(4) = 33320, \quad b_6(4) =5,$$
so $a_6(4) -b_6(4) = 33315 \equiv 3 \pmod 6 \not= 0.$\\
\emph{Perspective 2} $(p=3)$. Write $6=3^1\cdot 2$, so $a=1$, $m=2$. Corollary \ref{corofailure} now predicts failure at $n\leq p^{2a} =9.$ Modulo $3$,
$$F_{36}(x) \equiv \prod_{n\geq 1} (1-x^{9n})^4 \pmod 3,$$
so
$$[x^9]F_{36}(x) \equiv [y^1](1-y)^4 = -4 \equiv 2 \pmod 3.$$
Also,
$$C_6(x) \equiv \prod_{n\geq 1} (1-x^{18n})^2 \pmod 3,$$
hence $[x^9]C_6(x) =0.$ Therefore,
$$[x^9] (F_36(x) - C_6(x)) \equiv 2 \not\equiv 0\pmod 3,$$
so failure occurs at $n=9 \leq9$. Taking the minimum over both primes gives the sharp bound $n\leq 4.$
    
\end{example}

\section{Numerical Verification for $t=3$}

Let
\[
F_9(x)=\prod_{n\ge1}(1-x^n)^9=\sum_{n\ge0}a_3(n)x^n,\qquad
C_3(x)=\prod_{n\ge1}(1-x^{3n})^3=\sum_{n\ge0}b_3(n)x^n.
\]

Table 1 records the coefficients of $F_9(x)$ and $C_3(x)$ for selected $n$, together with their difference.

\[
\begin{array}{c|r|r|r}
\hline
n & a_3(n) & b_3(n) & a_3(n)-b_3(n)  \\ \hline
0 & 1 & 1 & 0  \\
1 & -8 & 1 & -9  \\
2 & 20 & 2 & 18  \\
3 & 0 & 0 & 0  \\
4 & -70 & 2 & -72 \\
5 & 64 & 1 & 63  \\
6 & 56 & 2 & 54  \\
8 & -125 & 1 & -126 \\ \hline
\end{array}
\]

\begin{center}
  \noindent\textbf{Table 1.} Coefficient comparison for $t=3$.  
\end{center}

The computations establish that $a_3(n)-b_3(n)$ is divisible by $9$ for each listed $n$, verifying the congruence $F_9(x)\equiv C_3(x)\pmod{9}$ in this range. This suggests that the prime power case $m=1$ may permit a lift of the congruence to higher $p$-adic orders, in sharp contrast to the non-prime-power case $m>1$, where the obstruction manifests already modulo $p$.

\section{Modular Form Perspective}
The generating functions are modular forms and eta-quotients:
$$C_t(x) = \frac{\prod_{n\geq 1} (1-x^{tn})^t}{\prod_{n\geq 1} (1-x^n)}, \quad F_{t^2}(x) = \prod_{n\geq 1} (1-x^n)^{t^2-1} = \eta(z)^{t^2 -1}|_{x = e^{2\pi iz}}.$$

In the Han-Ono paper \cite{Han-Ono} for $t=3$, 
\begin{itemize}
    \item $F_9(x) = {\eta(z)}^8 \in S_4(SL_2(\mathbb{Z}).$
    \item  $C_3(x) =  {\eta(9z)}^3/ \eta(3z)\in M_1(\Gamma_0(9), \chi).$
\end{itemize}
The congruence $F_{t^2}(x) \equiv C_t(x) \pmod t$ gives a congruence between modular forms of different levels and weights modulo $t$, which for prime $t=p$ is a shadow of the theory of $p$-adic modular forms and Hida families.

\section*{Acknowledgements}
The author would like to thank her advisor, Prof. Ben Kane, with whom she has had helpful discussions.


\begin{thebibliography}{1}\footnotesize

\bibitem{FDD} F. Garvan, D. Kim, and D. Stanton,\newblock{\em Cranks and $t$-cores}, \newblock{Invent. Math.} {\bf 101} (1990), no. 1, 1--17.

\bibitem{GN}  G.-N. Han, \newblock{\em Some conjectures and open problems on partition hook lengths}, \newblock{Experiment. Math.} {\bf 18} (2009), no. 1, 97--106.

\bibitem{GNHan}  G.-N. Han, \newblock{\em The Nekrasov-Okounkov hook length formula: Refinement, elementary proof, extensions and applications}, 
\newblock{ Ann. Inst. Fourier}  {\bf 60}(2010), 1--29.

\bibitem{Han-Ono}G.-N. Han and K. Ono, 
\newblock {\em Hook lengths and $3$-cores},
\newblock {Ann. Comb.} {\bf 15}(2011), 305--312.

\bibitem{Granville-Ono} A. Granville and K. Ono, \newblock{\em Defect zero $p$-blocks for finite simple groups}, \newblock{Trans. Amer. Math. Soc.} {\bf 348} (1996), no. 1, 331--347.

\bibitem{Kummer1852} E. E. Kummer, \newblock{\em Über die Ergänzungssätze zu den allgemeinen Reciprocitätsgesetzen}, \newblock{J. Reine Angew. Math.} {\bf 44} (1852), 93--146.


\bibitem{Nek-Oko}N. A. Nekrasov and A. Okounkov,
\newblock{\em Seiberg-Witten theory and random partitions},
{\em The Unity of Mathematics}, Progr. Math. {\bf 244}, Birkhäuser Boston, 2006,  525--596.



\end{thebibliography}
\end{document}